\documentclass{amsart}
\usepackage{amsfonts, amssymb}
\usepackage{amsmath}
\usepackage{amsthm}

\numberwithin{equation}{section}

\def\abs#1{\left| #1\right|}
\def\absval#1#2{\left|#1\right|_{#2}}

\newtheorem{theorem}{Theorem}[section]
\newtheorem{lemma}[theorem]{Lemma}

\newcommand{\Q}{\mathbb{Q}}
\newcommand{\C}{\mathbb{C}}
\newcommand{\N}{\mathbb{N}}

\title{Three approaches to detecting discrete integrability}

\author{R. G. Halburd and R. J.  Korhonen}
\begin{document}

\begin{abstract}
A class of discrete equations is considered from three perspectives corresponding to three measures of the complexity of solutions:
the (hyper-) order of meromorphic solutions in the sense of Nevanlinna, the degree growth of iterates over a function field and the height growth of iterates over the rational numbers.  In each case, low complexity implies a form of singularity confinement which results in a known discrete Painlev\'e equation.
\end{abstract}

\maketitle

\section{Introduction}
In this paper we will study the equation
    \begin{equation}\label{maineqD}
    y_{j+1}+y_{j-1}=\frac{a_j y_j^2+b_j y_j +c_j}{y_j^2},
    \end{equation}
where $c_j\not\equiv 0$,
from three different perspectives.  In each approach we will explore a different measure of complexity of solutions and we will interpret $y_j$ and the coefficients in a slightly different way.  Integrability has long been associated with the slow growth of complexity \cite{veselov:92}, however the first highly sensitive, yet heuristic, test for integrability of equations such as \eqref{maineqD} was the idea of singularity confinement \cite{grammaticosrp:91,ramanigh:91}.  In each of the three approaches studied here, we will use an analogue of singularity (non-) confinement, suitably re-interpreted, to get a lower bound on the relevant measure of complexity.

In analogy with the Painlev\'e property for differential equations, the idea behind singularity confinement is to study the behaviour of solutions at singular values of the dependent variable.  For equation \eqref{maineqD} we note that, if the coefficient functions are finite, the only way that $y_{j+1}$ can become infinite starting from finite values is if $y_j=0$.  In order to better understand this situation, we consider the perturbed initial conditions
$y_{j-1}=\kappa$, where $\kappa$ is arbitrary, and $y_j=\epsilon$ and expand the next few iterates as Laurent series about $\epsilon=0$.
If $c_j\ne 0$, this gives
\begin{equation}
\label{sing-conf}
\begin{split}
y_{j+1} &=\frac{c_j}{\epsilon^2}+\frac{b_j}{\epsilon}+O(1),
\\
y_{j+2} &=a_{j+1}-\epsilon+\frac{b_{j+1}}{c_j}\epsilon^2+O(\epsilon^3),
\\
y_{j+3} &=-\frac{c_j}{\epsilon^2}-\frac{b_j}{\epsilon}+\frac{(b_{j+2}a_{j+1}+c_{j+2})-\epsilon+O(\epsilon^2)}{\left(a_{j+1}-\epsilon+(b_{j+1}/c_j)\epsilon^2+O(\epsilon^3)\right)^2}+O(1).
\end{split}
\end{equation}
In the limit $\epsilon\to 0$, $y_{j+3}$ will be infinite unless $a_{j+1}=0$.  If $a_{j+1}=0$ we have
$$
y_{j+3}=\frac{c_{j+2}-c_j}{\epsilon^2}-\frac{b_{j+2}-2b_{j+1}+b_j}{\epsilon}+O(1).
$$
Taking the limit $\epsilon\to 0$ we see that $y_{j+3}=\infty$ unless 
\begin{equation}
\label{bcconds}
a_{j+1}=0,\quad
b_{j+2}-2b_{j+1}-b_j=0\quad\mbox{and}\quad
c_{j+2}-c_j=0.
\end{equation}
If these conditions are satisfied the singularity is said to be confined as $y_n$ remains finite (at least for the next few iterates).  Demanding that all singularities are confined in this way means that conditions \eqref{bcconds} must hold for all $j$. 
% This gives $b_j=\alpha j+\beta$ and $c_j=\gamma$.    
Hence equation \eqref{maineqD} becomes
\begin{equation}
\label{dp1-delta}
y_{j+1}+y_{j-1}=\frac{(\alpha j+\beta) y_j +(\gamma+\delta (-1)^j)}{y_j^2},
\end{equation}
where $\alpha$, $\beta$, $\gamma$ and $\delta$ are constants.  In order to avoid technicalities in some of the approaches that follow, we will restrict ourselves to the case in which
$a_j$, $b_j$ and $c_j$ are rational functions of $j$. This forces $\delta=0$ in equation (\ref{dp1-delta}), leaving us with the equation
\begin{equation}
\label{dp1-nodelta}
y_{j+1}+y_{j-1}=\frac{(\alpha j+\beta) y_j +\gamma}{y_j^2}.
\end{equation}
Equation (\ref{dp1-nodelta}) is known to have a continuum limit to the first Painlev\'e equation,  it is the compatibility condition for a related linear problem and it has been derived from the Schlesinger transformations of the third Painlev\'e equation \cite{fokasgr:93}.

Although singularity confinement has been successfully used to find many integrable discrete equations, Hietarinta and Viallet \cite{hietarintav:98} gave an example of an equation in which the singularities are confined and yet the dynamics appear to be chaotic.  
%{\bf 
%For simplicity, in this paper we will only consider the case in which the coefficient functions $a_j$, $b_j$ and $c_j$ are rational functions of $j$, so we must take $\delta=0$.}

The first of the three approaches to be considered in this paper is to study the growth, in the sense of Nevanlinna, of meromorphic solutions of difference equations.  
To this end we replace equation \eqref{maineqD} with its complex analytic version:
    \begin{equation}\label{maineq}
    y(z+1)+y(z-1)=\frac{a(z) y(z)^2+b(z) y(z) +c(z)}{y(z)^2},
    \end{equation}
where $a$, $b$ and $c$ are rational functions and $y$ is a non-rational meromorphic function.  It was suggested in \cite{ablowitzhh:00} that the existence of sufficiently many finite-order meromorphic solutions is a natural difference equation analogue of the Painlev\'e property.  
In \cite{halburdk:07PLMS} the authors used the existence of an admissible finite-order meromorphic solution to reduce a class of difference equations to a short list of difference Painlev\'e-type equations.  Here {\em admissible} means that the solution grows faster than the coefficients in a precise sense. In the case in which the coefficients are rational functions as considered here, this amounts to saying that the solution is non-rational.  In \cite{halburdkt:14} it was shown that the same conclusions remain valid if finite-order is replaced by hyper-order less than one.  

The order and hyper-order of a meromorphic function will be defined in section 2 of the present paper and we will prove the following.
\begin{theorem}\label{mainthm}
Let $y$ be a non-rational meromorphic solution of 
equation \eqref{maineq},
where $a$, $b$ and $c\not\equiv0$ are rational functions. If the hyper-order
of $y$ is less than one, then $a\equiv 0$, $b(z)=Az+B$ and $c(z)=C$, where $A$,
$B$ and $C$ are complex constants.
\end{theorem}
This is a special case of the classification in \cite{halburdk:07PLMS,halburdkt:14}.  By restricting ourselves to the case of rational coefficients, we eliminate many technical issues that arise from Nevanlinna theory, which allows us to concentrate on the role played by singularity confinement-type arguments in obtaining a lower bound on the hyper-order of solutions.  
%This allows us to present the central calculation very clearly without the need for considering many technicalities.  

In section 3 we will consider equation \eqref{maineqD} as a discrete equation (i.e. $j\in\mathbb{Z}$) but each $y_j$ is a rational function of an external complex variable $z$.  The natural measure of complexity here is the degree growth of the rational iterates $y_j$.  This is very close to the idea of algebraic entropy \cite{bellonv:99,hietarintav:98} in which one considers the degree $d_j$ of the $j^{\mbox{th}}$ iterate of equation \eqref{maineqD} as a rational function of $y_0$ and $y_1$.  The definition of the {\em algebraic entropy} is
$$
\lim_{j\to\infty}\frac{\log d_j}{j}.
$$
Zero algebraic entropy is associated with integrability.  By considering $y_0\equiv y_0(z)$ and $y_1\equiv y_1(z)$ as rational functions of $z$ we can use more elementary arguments based on complex analysis of a single variable.  It also gives us a more refined tool to consider one-parameter families of solutions, rather than considering the whole solution space at once.  If $y_0$ and $y_1$ are degree one rational functions, then the degree of $y_n$ will be the same as $d_j$ unless some cancellation has occurred on substitution into the expression for $y_n$ as a function of $y_0$ and $y_1$.  

Finally in section 4 we will consider the case in which the coefficients $a_j$, $b_j$, $c_j$ are rational functions of $j$ with rational coefficients and the solution of the discrete equation \eqref{maineqD}, $y_j\in\mathbb{Q}$ for all sufficiently large $j$.  In this setting the natural measure of complexity is the height.  The logarithmic height of $p/q$, where $p$ and $q$ are co-prime integers, is $h(p/q)=\log\max(|p|,|q|)$.  If the logarithmic height of all solutions grows polynomially, we say that the equation is {\em Diophantine integrable} \cite{halburd:05}.  This idea was suggested by applying the observation of Osgood \cite{osgood:81,osgood:84,osgood:85} and Vojta \cite{vojta:87} that there is a formal similarity between Nevanlinna theory and Diophantine approximation to the first approach to discrete integrability above.  The logarithmic height can be expressed as a sum over all of the (suitably normalised) absolute values on $\mathbb{Q}$ (i.e., the usual absolute value and the $p$-adic absolute values).  We will show how a calculation similar to the singularity confinement sequence \eqref{sing-conf} can be formulated in which the small quantity $\epsilon$ is small with respect to an arbitrary absolute value on $\mathbb{Q}$.  This then induces a lower bound on the height.
We will highlight the similarities between the previous two approaches and the proof of the following theorem, which appears in
\cite{halburdm} and the PhD thesis of Will Morgan \cite{morgan:10}.
\begin{theorem}
\label{DiophantineThm}
Let $a_n$, $b_n$ and $c_n\not\equiv 0$ be rational functions with coefficients in $\mathbb{Q}$.  Suppose that for sufficiently large $r_0$, $y_n\in\mathbb{Q}$  solves equation
\eqref{maineqD} for all $n\ge r_0$.  If
\begin{equation}
\label{admissible}
\sum_{n=r_0}^r \{h(a_n)+h(b_n)+h(c_n)\}=o\left(\sum_{n=r_0}^r h(y_n)\right)
\end{equation}
and 
\begin{equation}
\sum_{n=r_0}^r h(y_n)\le K r^\rho,
\label{Diobnd}
\end{equation}
for some positive constants $K$ and $\rho$, then equation \eqref{maineqD} reduces to equation \eqref{dp1-nodelta}.
\end{theorem}
A similar result leading to the discrete Painlev\'e II equation has been derived in \cite{alghassanih:15}.
Heights were first used in Abarenkova,  Angl{\`e}s d'Auriac,  Boukraa, Hassani and Maillard \cite{abarakovaabhm:99} to estimate entropy.   Heights have also been used to detect low complexity solutions in Silverman \cite{silverman:14} and Roberts and Vivaldi \cite{robertsv:15}.

\section{Existence of meromorphic solutions of hyper-order less than one}

In this section we will use the slow growth rate of meromorphic solutions to detect Painlev\'e type equations out of a natural class of second-order
difference equations. Our aim is to review the method of \cite{halburdk:07PLMS} by going through a simple case requiring as few technical details as possible. We will need a small number of concepts from Nevanlinna theory (see e.g. \cite{hayman:64}) such as the definitions of order, hyper-order and the counting function. Lemma \ref{nevlemma} below allows us to translate simple inequalities about the relative frequencies of zeros and poles into statements about the hyper-order.

Let $f$ be a meromorphic function in the complex plane. The {\em counting function} $n(r,f)$  is the number of poles of $f$ in the disc $\{z\in\C:|z|\leq r\}$, each pole counted according to its multiplicity. Moreover, we define $\log^+ x = \max\{\log x,0\}$ for any $x\geq 0$. Then
    $$
    N\left(r,f\right)=\int_0^r (n(t,f)-n(0,f))\frac{dt}{t} + n\left(0,f\right)\log r
    $$
is the {\em integrated counting function},
    $$
    m(r,f) = \int_0^{2\pi} \log^+ |f(re^{i\theta})|\frac{d\theta}{2\pi}
    $$
is the {\em proximity function}, and
    $$
    T(r,f) = m(r,f) + N(r,f)
    $$
is the {\em Nevanlinna characteristic function} of $f$. The order of $f$ is
    $$
    \sigma(f)=\limsup_{r\to\infty} \frac{\log T(r,f)}{\log r},
    $$
and the {\em hyper-order} is
    $$
    \varsigma(f)=\limsup_{r\to\infty} \frac{\log\log T(r,f)}{\log r}.
    $$

Note that by restricting the coefficients of \eqref{maineq} to be rational functions rules out the most general form of the difference Painlev\'e I equation, where $c$ is a period two function. The general case, was recovered in \cite{halburdk:07PLMS}.

\begin{proof}[Proof of Theorem~\ref{mainthm}: ] We will first show that if $|z|$ is sufficiently large, then each zero $\hat z$ of a non-rational meromorphic solution $y(z)$ of \eqref{maineq}  may be uniquely associated with a finite number of poles and
zeros of neighboring iterates $y(z\pm 1)$, $y(z\pm 2)$, $y(z\pm 3)$, $y(z\pm 4)$ such that the number of zeros
divided by the number of poles, both counting multiplicities, is
bounded by $4/5$ in each such grouping.

Let $y(z)$ be a meromorphic solution of \eqref{maineq} and assume first that $a(z)$ is not identically zero. Therefore, $a$, $b$ and $c$
have finitely many zeros (unless $b\equiv0$) and poles, since
they are rational functions of $z$, and so there exists an
$r_0\geq 0$ such that $a(z)\ne 0$ and $c(z)\ne 0$ for all $z$ satisfying
$|z|\geq r_0$.

Suppose that the solution $y(z)$ has a zero of multiplicity $k$ at some point $\hat{z}$ in the complex plane. Then $y(z)$ can be expressed as a Laurent series expansion
    \begin{equation}\label{laurent1}
     y(z) = \alpha (z-\hat{z})^{k}+\mathrm{O}((z-\hat{z})^{k+1}),\qquad \alpha\in\C\setminus\{0\},
    \end{equation}
in a sufficiently small  neighborhood of $\hat{z}$. All except finitely many zeros $\hat{z}$ of $y(z)$ satisfy $|\hat{z}|\geq r_0+1$, and so, for these zeros it follows that $c(\hat z+\sigma)\not=0$, where $\sigma=\pm1$. Now it follows by substituting \eqref{laurent1} to the equation \eqref{maineq} that $y(z+\sigma)$ has a pole at $\hat{z}$ of
order at least $2k$ for $\sigma=1$ or $\sigma=-1$. The order of the pole at $\hat{z}+\sigma$ can be strictly greater than $2k$ only if there is a pole of the same order $l>2k$ at both points $\hat{z}+1$ and $\hat{z}-1$. In this case we can find even more poles per zero compared to the case where the order of the pole at $\hat{z}+\sigma$ is equal to $2k$. Therefore, without loss of generality, we may assume that $y(z+\sigma)$ has a pole of order $2k$ at $\hat z$, and so
    \begin{equation}\label{laurent2}
      y(z+\sigma) = \beta_\sigma (z-\hat{z})^{-2k} + \mathrm{O}({(z-\hat{z})^{1-2k}}), \qquad \beta_\sigma\in\C\setminus\{0\}
    \end{equation}
for all $z$ in a small enough neighborhood of $\hat{z}$. Now, by shifting equation \eqref{maineq}, we obtain
    \begin{equation}\label{maineq_shift}
    y(z+2\sigma)+y(z)=\frac{a(z+\sigma) y(z+\sigma)^2+b(z+\sigma) y(z+\sigma) +c(z+\sigma)}{y(z+\sigma)^2}.
    \end{equation}
By substituting the Laurent series expansions \eqref{laurent1} and \eqref{laurent2} into \eqref{maineq_shift}, we have
    \begin{equation}\label{laurent3}
    y(z+2\sigma)= a(z+\sigma)  + \mathrm{O}((z-\hat{z})^{k})
    \end{equation}
in a neighborhood of $\hat{z}$. By continuing in this way, and summarizing the above, it follows that
    \begin{equation}\label{seq1}
    \begin{split}
    y(z) &= \alpha (z-\hat{z})^{k}+\mathrm{O}((z-\hat{z})^{k+1})\\
    y(z+\sigma) &= \beta_\sigma (z-\hat{z})^{-2k} + \mathrm{O}({(z-\hat{z})^{1-2k}})\\
    y(z+2\sigma)&= a(z+\sigma)  + \mathrm{O}((z-\hat{z})^{k})
    \\
    y(z+3\sigma) &=  - \beta_\sigma (z-\hat{z})^{-2k}  + \mathrm{O}((z-\hat{z})^{1-2k})
    \\
    y(z+4\sigma)&= a(z+3\sigma)-a(z+\sigma) + \mathrm{O}((z-\hat{z})^{k})
    \\
    \end{split}
    \end{equation}
where $\alpha$ and $\beta_{\sigma}$ are non-zero. Therefore, the zero of $y(z)$ of order $k$ may be grouped together with the pole of
$y(z+\sigma)$ of order $2k$. Note that even if
$a(z+3\sigma)-a(z+\sigma)=0$ we are free to associate the
available pole of $y(z+3\sigma)$ with the zero of
$y(z+4\sigma)$, if the zero is of order $k$ at most. If the zero of
$y(z+4\sigma)$ is of order $l>k$, then $\hat{z}+4$ is a starting point of another sequence of the type \eqref{seq1}. Hence the number of zeros divided by the number
of poles in the sequence \eqref{seq1} is less than or equal to 1/2 counting multiplicities, provided that $a\not\equiv0$ in \eqref{maineq}.

Suppose now that $a\equiv0$ so that equation \eqref{maineq}
reduces to
    \begin{equation}\label{mainaj0}
    y(z+1)+y(z-1)=\frac{b(z) y(z) +c(z)}{y(z)^2}.
    \end{equation}
We will again consider the case where $y(z)$ has a zero of order
$k$ at $z=\hat{z}$ assuming first that both $y(z+\sigma)$ and
$y(z-\sigma)$ have a pole at least of order $2k$. Then, as before, we may assume without loss of generality that the order of the pole is exactly $2k$, and so, by iterating \eqref{mainaj0}, it follows that
   \begin{equation}\label{seq2a}
    \begin{split}
    y(z) &= \alpha (z-\hat{z})^{k}+\mathrm{O}((z-\hat{z})^{k+1})\\
    y(z+\sigma) &= \beta_\sigma (z-\hat{z})^{-2k} + \mathrm{O}({(z-\hat{z})^{1-2k}})\\
    y(z+2\sigma)&=-\alpha (z-\hat{z})^{k}+\mathrm{O}((z-\hat{z})^{k+1})
    \\
    y(z+3\sigma) &=  - \beta_\sigma (z-\hat{z})^{-2k}  + \mathrm{O}((z-\hat{z})^{1-2k})
    \\
    y(z+4\sigma)&= \alpha (z-\hat{z})^{k}+\mathrm{O}((z-\hat{z})^{k+1}),
    \\
    \end{split}
    \end{equation}
where $\alpha\not=0$ and $\beta_{\sigma}\not=0$. In this sequence
the number of zeros divided by the number of poles is less than
or equal to 3/4, when multiplicities are taken into account.

Suppose now that $y(z)$ has a zero of order $k$ at $z=\hat{z}$ and
$y(z-\sigma)$ has a pole of order $l$ such that $k\leq l<2k$.
Then, by \eqref{mainaj0},
    \begin{equation}\label{seq2b}
    \begin{split}
    y(z-\sigma) &= \beta_\sigma (z-\hat{z})^{-l}+\mathrm{O}({(z-\hat{z})^{1-l}})\\
    y(z) &= \alpha (z-\hat{z})^k+\mathrm{O}((z-\hat{z})^{k+1})\\
    y(z+\sigma) &=  \frac{c(z)}{\alpha^2} (z-\hat{z})^{-2k}
    + \mathrm{O}({(z-\hat{z})^{1-2k}})\\
    y(z+2\sigma) &=-\alpha (z-\hat{z})^k+\mathrm{O}((z-\hat{z})^{k+1})
    \\
    y(z+3\sigma) &=  \frac{c(z+2\sigma)-c(z)}{\alpha^2} (z-\hat{z})^{-2k}  + \mathrm{O}((z-\hat{z})^{1-2k})
    \\
    \end{split}
    \end{equation}
where $\alpha\not=0$ and $\beta_\sigma\not=0$. The number of zeros
divided by the number of poles of $y$ in the set $\{\hat z - \sigma, \hat z, \hat z + \sigma, \hat z + 2\sigma \}$  is less than or
equal to $2/3$. It may happen that there is a zero of $y$ at $\hat z+3\sigma$, or at $\hat z+4\sigma$, but then this zero becomes a starting point of another sequence of one of the types \eqref{seq2a} or \eqref{seq2b}, or \eqref{seq2c} below. If there is another sequence of the type \eqref{seq2b} progressing in the opposite direction from the point $\hat z - \sigma$, then the corresponding zero-pole ratio of the two combined sequences in the set $\{\hat z - 4\sigma,\ldots, \hat z + 2\sigma \}$ is less than or equal to $4/5$.

We still need to deal with the case where $y(z)$ has a zero of
order $k$ and $y(z-\sigma)$ has a pole of order $l<k$ (or it
assumes a finite value.) The iteration of equation (\ref{mainaj0})
yields
    \begin{equation}\label{seq2c}
    \begin{split}
    y(z&+\sigma) = c(z) y(z)^{-2}+b(z) y(z)^{-1} -y(z-\sigma)\\
    y(z&+2\sigma)= -y(z)    + \frac{b(z+\sigma)}{c(z)}
    y(z)^2-\frac{b(z+\sigma)b(z)}{c(z)^2}y(z)^3+\mathrm{O}(y(z)^4)
    \\
    y(z&+3\sigma) =  (c(z+2\sigma)-c(z))y(z)^{-2} \\
    &+\left(-b(z)+\frac{2c(z+2\sigma)b(z+\sigma)}{c(z)}-b(z+2\sigma)\right)y(z)^{-1}  +y(z-\sigma) +O(1).
    %\\ &
    %+c(z+2\sigma)\left(\frac{-2b(z+\sigma)(b(z)-b(z+\sigma))}{c(z)^2}+\frac{b(z+\sigma)^2}{c(z)^2}\right)
    %\frac{-b(z+2\sigma)b(z+\sigma)}{c(z)} \\ &+\mathrm{O}(y(z)).
    \\
    \end{split}
    \end{equation}
There are $2k$ zeros (by taking $y(z)$ into account) and $4k$ poles
in the sequence \eqref{seq2c}, unless
    \begin{equation}
    \label{abeqns}
    c(z+2\sigma)-c(z)=0\mbox{  and
    }b(z+2\sigma)-2b(z+\sigma)+b(z)=0.
    \end{equation}
(To be exact, there are only $3k$ poles in \eqref{seq2c} if the first of the equations in \eqref{abeqns} holds and the second one doesn't. In this case the zero-pole ratio in this sequence is $2/3$.) Now unless equations (\ref{abeqns}) hold for all $z$ then at least
one of them will fail to hold for all sufficiently large $|z|$. If
these equations hold for all $z$ then \eqref{abeqns} become linear difference equations. Solving these equations, and taking into account that the coefficients $b(z)$ and $c(z)$ are rational functions by assumption, it follows that $b(z)=Az+B$ and $c(z)=C$ for complex constants $A$, $B$ and $C$. In this case, the proof is complete. Otherwise,
we have been able to associate each zero of $y(z)$ with an appropriate number of zeros
and poles of ``nearby'' iterates $y(z\pm1),\ y(z\pm2),\ y(z\pm3),\ y(z\pm4)$ for all sufficiently large
$|z|$ such that within each grouping the number of zeros divided by
the number of  poles is at most $4/5$. Therefore,
    \begin{equation}\label{n_inequality}
    n\left(r,\frac1y\right) \leq \frac45 n(r+2,y) + O(1).
    \end{equation}
Lemma~\ref{nevlemma} below, which is a special case of \cite[Lemma 2.1]{halburdk:PAMS}, applied to \eqref{n_inequality} with $a=0$, $s=2$ and $\tau=4/5$, implies that  the hyper-order of $y$ is at least one.

\begin{lemma}[\cite{halburdk:PAMS}]\label{nevlemma}
Let $f(z)$ be a non-rational meromorphic solution of
    \begin{equation}\label{Peq}
    P(z,f)=0
    \end{equation}
where $P(z,f)$ is difference polynomial in $f(z)$ with rational coefficients, and let $a\in\C$ satisfy $P(z,a)\not\equiv 0$. If there exists $s>0$ and $\tau\in(0,1)$ such that
    \begin{equation}\label{n_inequality_lemma}
    n\left(r,\frac{1}{f-a}\right) \leq \tau\, n(r+s,f) + O(1),
    \end{equation}
then the hyper-order $\varsigma(f)$ of $f$ is at least $1$.
\end{lemma}

We conclude that the only possible case when non-rational solutions of hyper-order less than one can exist is when $b(z)=Az+B$ and $c(z)=C$.
\end{proof}

\section{Polynomial degree growth of rational iterates}

%Discrete Painlev\'e type equations can be detected in a number of
%different methods, many of which reiterate Veselov's observation
%that the integrability of mappings correlate with weak growth of
%certain characteristics \cite{veselov:92}. In some approaches a
%discrete equation is considered to be integrable if the degree of
%the $n^{\textrm{th}}$ iterate of a discrete equation, considered
%as a rational function of the initial conditions, is bounded by a
%power of $n$ \cite{falquiv:93,bellonv:99}. Hietarinta and Viallet \cite{hietarintav:98} suggested
%that integrability of discrete equations correlates with the \textit{algebraic entropy}
%    \begin{equation*}
%    h=\lim_{n\to\infty}\frac{1}{n}\log\deg(d_n)
%    \end{equation*}
%of their solutions $\{d_n\}_{n\in\N}$ satisfying $h=0$.

%deal with sequences that consists of zeros
In this section we will prove the following.
\begin{theorem}\label{mainthm2}
Let $\{y_j\}_{j\in\N}$ be a sequence of non-constant rational
functions of $z$ solving equation \eqref{maineqD}
%    \begin{equation}\label{maineq2}
 %   y_{j+1}+y_{j-1}=\frac{a_j y_j^2+b_j y_j +c_j}{y_j^2},
 %   \end{equation}
where $a_j$, $b_j$ and $c_j\not\equiv0$ are rational
functions of $j$. If the degree of $\{y_j\}_{j\in\N}$ grows at most polynomially in $j$, then $a_j=0$, $b_j=Aj+B$ and $c_j=C+D(-1)^j$, where $A$,
$B$, $C$ and $D$ are constants.
\end{theorem}

%Equation \eqref{maineq2} has been historically named as an
%alternate discrete Painlev\'e I equation based on the fact that is
%is one of those discrete Painlev\'e equations which possess a
%continuum limit to the first Painlev\'e (differential) equation
%\cite{grammaticosr:93,ramanig:96}, and it is integrable in the
%sense that it possesses a Lax pair \cite{fokasgr:93}. Nowadays the
%labelling of these equations is mostly done based on Sakai's
%classification using their affine Weyl group structure
%\cite{sakai:01}.

The key idea behind the proof is to use the fact that the degree of a rational function is the number of zeros or poles (counting multiplicities) in $\mathbb{CP}^1$.  We will use singularity confinement-type calculations similar to \eqref{sing-conf} to relate the number of zeros and poles of nearby iterates.

\vskip 3mm

\textit{Proof of Theorem \ref{mainthm2}: } The first part of the proof is largely analogous to the first part of the proof of Theorem~\ref{mainthm}, and it consists of determining the relative zero and pole densities of the solution sequence. We assume first that $a_j$ is not identically zero. Therefore, $a_j$, $b_j$ and $c_j$
have finitely many zeros (unless $b_j\equiv0$) and poles, since
they are rational functions of $j$, and so there exists a
$j_0\geq 0$ such that $c_j\ne 0$ for all $j$ satisfying
$|j|\geq j_0$.

We will show that if $|j|$ is sufficiently large, then each zero
of $y_j$ may be uniquely associated with a finite number of poles and
zeros of neighboring iterates $y_i$ such that the number of zeros
divided by the number of poles, both counting multiplicities, is
bounded by $4/5$ in each obtained grouping.

Suppose that the rational function $y_j$ has a zero of
multiplicity $k$ at some point $\hat{z}$ on the complex sphere. By
using a M\"obius transformation, if necessary, we may assume
without loss of generality that $\hat{z}=0$. In the following,
expressions such as ``$y_{j+1}$ has a pole'' will refer to a pole
at $z=0$.  Since $y_j$ has a zero of order $k$, then it follows
from equation \eqref{maineqD} that $y_{j+\sigma}$ has a pole of
order at least $2k$ for $\sigma=1$ or $\sigma=-1$. Since we have
taken $j$ to be sufficiently large, it follows that
$a_{j+\sigma}\not=0$, and so iteration of equation \eqref{maineqD}
gives
    \begin{equation}\label{seq1d}
    \begin{split}
    y_{j} &= \alpha z^{k}+\mathrm{O}(z^{k+1})\\
    y_{j+\sigma} &= \beta_\sigma z^{-2k} + \mathrm{O}({z^{1-2k}})\\
    y_{j+2\sigma}&= a_{j+\sigma}  + \mathrm{O}(z^{k})
    \\
    y_{j+3\sigma} &=  - \beta_\sigma z^{-2k}  + \mathrm{O}(z^{1-2k})
    \\
    y_{j+4\sigma}&= a_{j+3\sigma}-a_{j+\sigma} + \mathrm{O}(z^{k})
    \\
    \end{split}
    \end{equation}
where $\alpha$ and $\beta_{\sigma}$ are non-zero. Therefore, the
zero of $y_{j}$ of order $k$ may be associated with the pole of
$y_{j+\sigma}$ of order $2k$. Note that even if
$a_{j+3\sigma}-a_{j+\sigma}=0$, then we can associate the
available pole of $y_{j+3\sigma}$ with this zero of
$y_{j+4\sigma}$ (or the zero is a starting point of a new sequence of iterates of the type \eqref{seq1d} in a similar way as in the case \eqref{seq1} of meromorphic solutions). Hence the number of zeros divided by the number
of poles in sequence \eqref{seq1d} is less than or equal to 1/2
(counting multiplicities) under the assumption that $a_j\not\equiv0$.

Suppose now that $a_j\equiv0$ so that equation \eqref{maineqD}
reduces to
    \begin{equation}\label{mainaj0d}
    y_{j+1}+y_{j-1}=\frac{b_j y_j +c_j}{y_j^2}.
    \end{equation}
We will again consider the case where $y_j$ has a zero of order
$k$ at $z=0$ assuming first that both $y_{j+\sigma}$ and
$y_{j-\sigma}$ have a pole at least of order $2k$. Then, by
iterating \eqref{mainaj0d}, it follows that
   \begin{equation}\label{seq2ad}
    \begin{split}
    y_{j} &= \alpha z^{k}+\mathrm{O}(z^{k+1})\\
    y_{j+\sigma} &= \beta_\sigma z^{-2k} + \mathrm{O}({z^{1-2k}})\\
    y_{j+2\sigma}&=-\alpha z^{k}+\mathrm{O}(z^{k+1})
    \\
    y_{j+3\sigma} &=  - \beta_\sigma z^{-2k}  + \mathrm{O}(z^{1-2k})
    \\
    y_{j+4\sigma}&= \alpha z^{k}+\mathrm{O}(z^{k+1}),
    \\
    \end{split}
    \end{equation}
where $\alpha\not=0$ and $\beta_{\sigma}\not=0$. In this sequence
the number of zeros divided by the number of poles is less than or
equal to 3/4, when multiplicities are taken into account.

Suppose now that $y_j$ has a zero of order $k$ at $z=0$ and
$y_{j-\sigma}$ has a pole of order $l$ such that $k\leq l<2k$.
Then, by \eqref{mainaj0d},
    \begin{equation}\label{seq2bd}
    \begin{split}
    y_{j-\sigma} &= \beta_\sigma z^{-l}+\mathrm{O}({z^{1-l}})\\
    y_j &= \alpha z^k+\mathrm{O}(z^{k+1})\\
    y_{j+\sigma} &=  \frac{c_{j}}{\alpha^2} z^{-2k}
    + \mathrm{O}({z^{1-2k}})\\
    y_{j+2\sigma} &=-\alpha z^k+\mathrm{O}(z^{k+1})
    \\
    y_{j+3\sigma} &=  \frac{c_{j+2\sigma}-c_j}{\alpha^2} z^{-2k}  + \mathrm{O}(z^{1-2k})
    \\
    \end{split}
    \end{equation}
where $\alpha\not=0$ and $\beta_\sigma\not=0$. The number of zeros
divided by the number of poles in \eqref{seq2bd} is less than or
equal to $2/3$, provided that $y_{j+3\sigma}$ is non-zero. If
$y_{j+3\sigma}$ vanishes, the sequence of iterates starting from
$y_{j+3\sigma}$ of \eqref{seq2bd} becomes a special case of the
sequence \eqref{seq2cd} below. If there are two sequences of the type \eqref{seq2bd} joined together in a similar way as in the meromorphic case \eqref{seq2b}, then the number of zeros
divided by the number of poles in the combined sequence is less than or
equal to $4/5$.

We still need to deal with the case where $y_j$ has a zero of
order $k$ and $y_{j-\sigma}$ has a pole of order $l<k$ (or it
assumes a finite value.) The iteration of equation (\ref{mainaj0d})
yields
    \begin{equation}\label{seq2cd}
    \begin{split}
    y_{j+\sigma} &= c_j y_j^{-2}+b_j y_j^{-1} -y_{j-\sigma}\\
    y_{j+2\sigma}&= -y_j    + \frac{b_{j+\sigma}}{c_j}
    y_j^2-\frac{b_{j+\sigma}b_j}{c_j^2}y_j^3+\mathrm{O}(y_j^4)
    \\
    y_{j+3\sigma} &=  (c_{j+2\sigma}-c_j)y_j^{-2}
    +\left(-b_j+\frac{2c_{j+2\sigma}b_{j+\sigma}}{c_j}-b_{j+2\sigma}\right)y_j^{-1} +y_{j-\sigma} +\mathrm{O}(1).
    %\\ &\quad
    %+c_{j+2\sigma}\left(\frac{-2b_{j+\sigma}(b_j-b_{j+\sigma})}{c_j^2}+\frac{b_{j+\sigma}^2}{c_j^2}\right)
   % \frac{-b_{j+2\sigma}b_{j+\sigma}}{c_j}+\mathrm{O}(y_j).
    \\
    \end{split}
    \end{equation}
There are $2k$ zeros (by taking $y_j$ into account) and $4k$ (or $3k$) poles
in the sequence \eqref{seq2cd}, unless
    \begin{equation}
    \label{abeqnsd}
    c_{j+2\sigma}-c_j=0\mbox{  and
    }b_{j+2\sigma}-2b_{j+\sigma}+b_j=0.
    \end{equation}
Now unless equations (\ref{abeqnsd}) hold for all $j$ then at least
one of them will fail to hold for all sufficiently large $j$. If
these equations hold for all $j$ then $b_j=Aj+B$ and
$c_j=C+D(-1)^j$ for constants $A$, $B$, $C$ and $D$.  Otherwise,
note that $\deg_z y_j$ is the total number of zeros of $y_j$ on
the complex sphere counting multiplicities and it is also the
total number of poles of $y_j$. Since we have been able to
associate each zero of $y_j$ uniquely with an appropriate number of zeros
and poles of iterates $y_i$ close to $y_j$ for all sufficiently large
$j$ such that within each grouping the number of zeros divided by
the number of  poles is less than or equal to $4/5$, we have
    \begin{equation*}
    D_n(y_j)\le \frac45 D_{n+s}(y_j)+O(1),\qquad n\to\infty,
    \end{equation*}
for some $s>0$, where
    \begin{equation*}
    D_n(y_j):=\sum_{j=-n}^n \deg_z(y_j).
    \end{equation*}
It follows that
    \begin{equation*}
    \limsup_{n\to\infty}\frac{1}{n} \log D_n(y_j)\geq
    \log\frac{5}{4}>0,
    \end{equation*}
so the degree growth of $\{y_j\}_{j\in\N}$ is exponential.

\section{Diophantine integrability}
In this section we will consider the solution $(y_j)$ of equation \eqref{maineqD} to be a sequence of rational numbers and we will explore the growth of the height of such solutions subject to the assumptions of Theorem~\ref{DiophantineThm}.
We denote the usual absolute value  on $\mathbb{Q}$ by $|\cdot |_\infty$.  Any nontrivial absolute value on $\mathbb{Q}$ is equivalent to $|\cdot |_\infty$ or to one of the $p$-adic absolute values $|\cdot |_p$, for some prime $p$.  Given a prime $p$, any non-zero $x\in\mathbb{Q}$ can be written as $x=\frac ab p^r$ for $a,\,b,\,r\in\mathbb{Z}$, where $p\not |ab$.  Then the $p$-adic absolute value of $x$ is defined to be $|x|_p=p^{-r}$.  
The $p$-adic absolute values are non-Archimedean, which means they satisfy the strong triangle inequality 
$$
|x+y|_p\le \max\left\{|x|_p,|y|_p\right\},
$$ 
for all $x$ and $y\in\mathbb{Q}$.

An important identity for our calculations is the following expression for the logarithmic height %$h(x)$ of $x=a/b$, where $a$ and $b$ are co-prime, is
in terms of absolute values
$$
h(x)=\sum_{p\le\infty}\log^+|x|_p,
$$
where the sum is over all primes $p$ as well as the ``prime at infinity'' $p=\infty$.

We begin by fixing a particular absolute value $|\cdot|_p$ on $\mathbb{Q}$ and we use this to determine a certain length scale $\epsilon_n$ for the $n$th iterate.  Since the coefficient functions are rational functions of $n$, then for sufficiently large $n$, they are either identically zero or they are finite and non-zero.  From now on we work with sufficiently large $n$ in this sense.  For some $0<\delta< 1/2$
we define $\epsilon_n$ by
\begin{equation}
\label{epsilon1}
\begin{split}
\epsilon_n^{-\delta} = \kappa_p & \max\{1, \absval{c_n}{p}^{-1}, \absval{b_n}{p}, \absval{a_n}{p},\\
& \absval{c_{n+1}}{p}, \absval{c_{n-1}}{p}, \absval{b_{n+1}}{p}, \absval{b_{n-1}}{p}, \absval{a_{n+1}}{p}^{-1},\absval{a_{n-1}}{p}^{-1}\},
\end{split}
\end{equation}
where  $\kappa_p = 1$ $\forall$  $p < \infty$ and $\kappa_\infty = 3$.
The following lemma, the proof of which is elementary and can be found in \cite{halburdm}, relates small and large values of $y_m$ with respect to the given absolute value.

\begin{lemma}\label{DiophantineLemma1} 
Fix a prime $p\ \leq \ \infty$. Let $(y_n)$ be a solution to equation (\ref{maineqD}) where $a_nc_n\not\equiv 0$. Suppose that for some integer $m$, $\absval{y_m}{p} < \epsilon_m$, where $\epsilon_m$ is defined by
\eqref{epsilon1}.  Then either 
$$\absval{y_{m+1}}{p} \geq \frac{1}{\absval{y_m}{p}^{2-\delta}} \hbox{ and } \absval{y_{m+2}}{p} \geq \epsilon_{m+2},$$
or
$$\absval{y_{m-1}}{p} \geq \frac{1}{\absval{y_m}{p}^{2-\delta}} \hbox{ and } \absval{y_{m-2}}{p} \geq \epsilon_{m-2}.$$
\end{lemma}

\begin{lemma}
If $(y_j)$ is a solution of equation \eqref{maineqD} satisfying the assumptions of Theorem \ref{DiophantineThm}, then $\gamma_n\equiv 0$.
\end{lemma}

\begin{proof}
Assume that $\gamma_n\not\equiv 0$.  Let $|\cdot|$ denote the absolute value corresponding to the prime $p$. For sufficiently large $r_0$, define
\begin{eqnarray*}
S_1(r) & = & \{ n : r_0 \leq n \leq r \hbox{ and } \abs{y_n} < \epsilon_n\}\\
S_2(r) & = & \{ n : r_0 \leq n \leq r \hbox{ and } \abs{y_n} \geq \epsilon_n\},
\end{eqnarray*}
%where $r_0$ is to be taken as defined earlier e.g. the number that is two more than the maximum of the zeroes of the coefficients (and relevant combinations). 
Now  
\begin{equation}
\sum_{k=r_0}^{r}\log^+{\frac{1}{\abs{y_k}}} = \sum_{k \in S_1(r)}\log^+{\frac{1}{\abs{y_k}}} + \sum_{k \in S_2(r)}\log^+{\frac{1}{\abs{y_k}}}.\label{S1S2Sum} 
\end{equation}
%To estimate the term on the LHS examine the two terms on the RHS. First note that $$\logplus{\frac{1}{\abs{y_k}}} = 0 \hbox{ if } \abs{y_k} \geq 1.$$ 
Using Lemma \ref{DiophantineLemma1} gives
\begin{equation}
\sum_{k\in S_1(r)}\log^+{\abs{\frac{1}{y_k}}}  
 \leq \frac{1}{2-\delta} \sum_{k=r_0-1}^{r+1}\log^+{\abs{y_k}}.\label{S1bound}
\end{equation}
%The only non-zero terms in the sum over $S_2(r)$ are $k\in S_2(r)$ such that $\epsilon_k\leq\abs{y_k} <1$.  So
%Now the second term can be estimated in the following way,
Also
\begin{eqnarray*}
\sum_{k \in S_2(r)}\log^+{\frac{1}{\abs{y_k}}} & \leq & \sum_{k \in S_2(r)}\log^+{\epsilon_k^{-1}} \leq \sum_{k=r_0}^{r}\log^+{\epsilon_k^{-1}} \\
& = & \frac{1}{\delta}\sum_{k=r_0}^{r}\log^+{}\left( \kappa_p\max\{1, \abs{c_k}^{-1}, \abs{b_k}, \abs{a_k}, \abs{c_{k+1}}, \abs{c_{k-1}}, \right.\\ 
& & \left.  \abs{b_{k+1}}, \abs{b_{k-1}}, \abs{a_{k+1}}^{-1},\abs{a_{k-1}}^{-1}\}\right) \\
& \leq & \frac{1}{\delta}\left((r - r_0)\log^+{\kappa_p} + \sum_{k=r_0 - 1}^{r+1}\left[ \log^+{\abs{c_k}^{-1}} + 3\log^+{\abs{b_k}} \right. \right. \nonumber\\
&&+\left. \left. \log^+{\abs{a_k}} + 2\log^+{\abs{c_k}} + 2\log^+{\abs{a_k}^{-1}}\right]\right) \label{S2bound},
\end{eqnarray*}
where if $\abs{b_k} = 0$ it is excluded from the list.

Recall that $$ h(x) = \sum_{p\leq \infty}\log^+{\absval{x}{p}} = h\left(\frac{1}{x}\right) = \sum_{p\leq\infty}\log^+{\absval{\frac{1}{x}}{p}}.$$
So taking the sum over all primes $p\le\infty$  in equation (\ref{S1S2Sum}) gives
\begin{equation}
\sum_{k = r_0}^{r}h(y_k) \leq \frac{1}{2-\delta}\sum_{k = r_0-1}^{r+1}h(y_k) + \frac{1}{\delta}\left(3\sum_{k=r_0-1}^{r+1}(h(\alpha_k)+h(\beta_k) + h(\gamma_k)) + (r-r_0)\log{3}\right).
\end{equation}
So from \eqref{admissible}, we have $$\sum_{k = r_0}^{r}h(y_k) \leq \frac{1}{2-\delta}\sum_{k = r_0-1}^{r+1}h(y_k) + o\left(\sum_{k = r_0}^{r}h(y_k)\right),$$
which is impossible if $y_k$ satisfies \eqref{Diobnd}.
\end{proof}

\vskip 3mm

We conclude this section by quoting another lemma from \cite{halburdm} which bears a strong similarity to the singularity confinement-type calculation in \eqref{sing-conf}.  We do not give the precise definition of $\epsilon_k$ here but merely note that it depends on the absolute values of various combinations of the coefficient functions.
\begin{lemma} Let $(y_n)_{n=k-1}^{k+3} \subseteq \Q /\{0\}$ with $k-1 \geq r_0$ satisfy \begin{equation}y_{n+1} + y_{n-1} = \frac{c_n + b_n y_n}{y_n^2}.\label{singconeq}\end{equation} If $\abs{y_{k-1}} \leq \abs{y_{k}}^{-1/2}$ and, for sufficiently small $\delta > 0$, $\abs{y_k} < \epsilon_k$ then
\begin{enumerate}
\item $y_{k+1} = \frac{c_k}{y_k^2} + \frac{b_k}{y_k} + A_k,$ where $\abs{A_k}\leq \abs{y_k}^{-1/2}$.\\
\item $y_{k+2} = -y_k + \frac{b_{k+1}}{c_k}y_k^2 + B_k,$ where $\abs{B_k} \leq \abs{y_k}^{3-4\delta}$\\
\item $y_{k+3} = \frac{c_{k+2} - c_k}{y_{k+2}^2} + \frac{b_{k+2}- 2 \frac{c_{k+2}}{c_k}b_{k+1} + b_k}{y_{k+2}} + C_k$, where
$$\abs{C_k} \leq \max\left\{\abs{\frac{c_{k+2} - c_k}{c_k}}\abs{y_{k+2}}^{1-\delta}, \abs{y_{k+2}}^{-1/2}\right\}$$
 for non-Archimedean absolute values and 
 $$\abs{C_k} \leq 2\abs{\frac{c_{k+2} - c_k}{c_k}}\abs{y_{k+2}}^{1-\delta} + 3\abs{y_{k+2}}^{-1/2}$$ 
 for Archimedean absolute values.\\
\item $\abs{y_{k+2}} = \abs{y_k}$ for non-Archimedean absolute values and $\frac{36}{25}\abs{y_k} > \abs{y_{k+2}} > \frac{16}{25}\abs{y_k}$ for Archimedean absolute values. 
\end{enumerate}
\label{singularityconfinementtheorem}
\end{lemma}

%\section{Discussion}

%\vskip 5mm
%\noindent{Acknowledgements}

\bibliographystyle{amsplain}
%\bibliography{database0713}

\providecommand{\href}[2]{#2}

\end{document}